\documentclass[12pt]{amsart}
\usepackage[left=80pt,right=80pt]{geometry}
\usepackage[usenames]{color}
\usepackage{amsmath}
\usepackage{amssymb}
\usepackage{amsthm}
\usepackage{enumitem}
\usepackage{float}
\usepackage{graphicx}
\usepackage{array}
\usepackage{tikz}
\usetikzlibrary{matrix}
\usetikzlibrary{arrows}

\usepackage{hyperref,url}
\hypersetup{
	colorlinks=true,
  linkcolor=black,          
  citecolor=black,         
  filecolor=black,      
  urlcolor=black           
}

\newtheorem{prop}{Proposition}
\newtheorem{lemma}[prop]{Lemma}
\newtheorem{theorem}[prop]{Theorem}

\theoremstyle{definition}

\newtheorem{example}[prop]{Example}

\newcommand{\N}{\mathbb{N}}

\newcommand{\seqnum}[1]{\href{https://oeis.org/#1}{\rm \underline{#1}}}
\allowdisplaybreaks
\newcommand{\mylabel}[2]{#2\def\@currentlabel{#2}\label{#1}}

\begin{document}
\tikzset{mystyle/.style={matrix of nodes,
        nodes in empty cells,
        row 1/.style={nodes={draw=none}},
        row sep=-\pgflinewidth,
        column sep=-\pgflinewidth,
        nodes={draw,minimum width=1cm,minimum height=1cm,anchor=center}}}
\tikzset{mystyleb/.style={matrix of nodes,
        nodes in empty cells,
        row sep=-\pgflinewidth,
        column sep=-\pgflinewidth,
        nodes={draw,minimum width=1cm,minimum height=1cm,anchor=center}}}

\title{Graph Labelings Obtainable by Random Walks}

\author{Sela Fried$^{*\dagger}$}
\thanks{$^{*}$ Corresponding author.}
\thanks{$^{\dagger}$ Department of Computer Science, Israel Academic College,
52275 Ramat Gan, Israel.
\\
\href{mailto:friedsela@gmail.com}{\tt friedsela@gmail.com}} 
\author{Toufik Mansour$^{\sharp}$}
\thanks{$^{\sharp}$ Department of Mathematics, University of Haifa, 3498838 Haifa,
Israel.\\
\href{mailto:tmansour@univ.haifa.ac.il}{\tt tmansour@univ.haifa.ac.il}}

\maketitle

\begin{abstract}
We initiate the study of what we refer to as random walk labelings of graphs. These are graph labelings that are obtainable by performing a random walk on the graph, such that the labeling occurs increasingly whenever an unlabeled vertex is encountered. Some of the results we obtain involve sums of inverses of binomial coefficients, for which we obtain new identities. In particular, we prove that $\sum_{k=0}^{n-1}2^{k}(2k+1)^{-1}\binom{2k}{k}^{-1}\binom{n+k}{k}=\binom{2n}{n}2^{-n}\sum_{k=0}^{n-1}2^{k}(2k+1)^{-1}\binom{2k}{k}^{-1}$, thus confirming a conjecture of Bala.

\bigskip


\smallskip

\noindent \textbf{Math.~Subj.~Class.:} 05C78, 05A10, 05A15, 05C81.
\end{abstract}

\section{Introduction}
Suppose that in a beehive the queen bee wanders about and whenever she detects an empty cell, she lays an egg in it. In how many ways can this be done, assuming that the process may begin at any cell? Formally, the process can be modeled as a graph labeling by a random walk. Figure \ref{fig:100} below gives an example using the path graph on $7$ vertices.

\begin{figure}[H]
\centering
\scalebox{1.00}{
\begin{tikzpicture}[shape = circle, node distance=4cm and 7cm, minimum size=0.5cm]
\node[draw] at (1, 0)   (1) {$7$};
\node[draw] at (2, 0)   (2) {$6$};
\node[draw] at (3, 0)   (3) {$5$};
\node[draw] at (4, 0)   (4) {$3$};
\node[draw] at (5, 0)   (5) {$2$};
\node[draw] at (6, 0)   (6) {$1$};
\node[draw] at (7, 0)   (7) {$4$};

\draw[-] (1) -- (2);
\draw[-] (2) -- (3);
\draw[-] (3) -- (4);
\draw[-] (4) -- (5);
\draw[-] (5) -- (6);
\draw[-] (6) -- (7);
\end{tikzpicture}
\hspace{0.5cm}
\begin{tikzpicture}[shape = circle, node distance=4cm and 7cm, minimum size=0.5cm]
\node[draw] at (1, 0)   (1) {$4$};
\node[draw] at (2, 0)   (2) {$6$};
\node[draw] at (3, 0)   (3) {$5$};
\node[draw] at (4, 0)   (4) {$3$};
\node[draw] at (5, 0)   (5) {$2$};
\node[draw] at (6, 0)   (6) {$1$};
\node[draw] at (7, 0)   (7) {$7$};

\draw[-] (1) -- (2);
\draw[-] (2) -- (3);
\draw[-] (3) -- (4);
\draw[-] (4) -- (5);
\draw[-] (5) -- (6);
\draw[-] (6) -- (7);
\end{tikzpicture}}
\caption{Of the two labelings of the path graph on $7$ vertices given here, only the left is obtainable by a random walk.}\label{fig:100}
\end{figure}
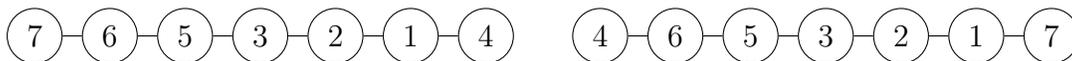

The literature on graph labeling is vast and many different kinds of labelings have been previously studied (see \cite{G} for a comprehensive literature survey). Nevertheless, we know of no works that are concerned with the question raised above. This work, that initiates the study of this question, is structured as follows: We begin with a short preliminary section, in which we define the concept of a random walk graph labeling and provide several elementary examples. In Section \ref{sec; 3}, that contains our main results, we calculate the number of random walk labelings of the King's graph and of the grid graph. The expression for the latter consists of a sum involving inverses of binomial coefficients. In the rest of the section we analyse this expression. In particular, we prove that
$$\sum_{k=0}^{n-1}\frac{2^{k}}{(2k+1)\binom{2k}{k}}\binom{n+k}{k}=\frac{\binom{2n}{n}}{2^{n}}\sum_{k=0}^{n-1}\frac{2^{k}}{(2k+1)\binom{2k}{k}},$$ thus confirming a conjecture of Bala.

\section{Preliminaries and first results}
Let $m,n\in\N$ be used throughout the work. We denote by $[n]$ the set $\{1,2,\ldots,n\}$. If $G$ is a graph with vertex set $V$ and edge set $E$, we write $G = (V, E)$ (we consider only simple graphs, i.e., graphs without multiple edges and loops). A \emph{labeling} of a graph $G = (V,E)$ of order $|V| = n$ is a bijection $\sigma\colon V\to [n]$. A \emph{random walk labeling} of $G$ is a labeling that is obtainable by performing the following process:
\begin{enumerate}
    \item Set $i=1$ and let $v\in V$. Define $\sigma(v)=i$.
    \item As long as there are vertices that are not labeled, pick $w\in V$ that is adjacent to $v$ and replace $v$ with $w$. If $v$ is not labeled, increase $i$ by $1$ and define $\sigma(v)=i$.
\end{enumerate} We denote by $\mathcal{L}(G)$ the number of random walk labelings of a graph $G$. Let us give a few easy examples:

\begin{example}\label{ex;110}
\begin{enumerate}
    \item [(a)]\label{a1} We have $\mathcal{L}(K_n)=n!$, where $K_n$ is the complete graph on $n$ vertices.
    \item [(b)]\label{b1} We have $\mathcal{L}(P_n)=2^{n-1}$ (\seqnum{A000079}), where $P_n$ is the path graph on $n$ vertices. Indeed, suppose that we start at the $k$th vertex, for $k\in[n]$. Then there are $k-1$ vertices on our left and $n-k$ vertices on our right. Since the vertices on either side must be labeled ``outwards", the random walk labeling is completely determined by the order in which we alternate between the two sides. We conclude that there are $\binom{n-1}{k-1}$ random walk labelings that begin at the $k$th vertex. Thus, $$\mathcal{L}(P_n)=\sum_{k=1}^n\binom{n-1}{k-1}=2^{n-1}.$$
    \item [(c)]\label{c1} We have $\mathcal{L}(C_n)=n2^{n-2}$ (\seqnum{A057711}), where $C_n$ is the cycle graph on $n$ vertices. Indeed, there are $n$ vertices at which we can start. Again, there are two sides between which we can alternate, but, this time, both sides are $n-1$ vertices long. Thus, as long as there are at least two unlabeled vertices, we have two possibilities at each step. This happens exactly $n-2$ times.
\end{enumerate}
\end{example}

In the next section, we consider more challenging graphs.

\section{Main results}\label{sec; 3}

\subsection{The king's graph}
The \emph{king's graph of size $m\times n$} is the graph whose vertices correspond to the $mn$ squares in a chessboard of size $m\times n$ and two vertices are adjacent if and only if a king can move between the corresponding squares in one move (cf.~\cite[p.~223]{C}). We denote by $KG_{m,n}$ the King's graph of size $m\times n$.
Calculating $\mathcal{L}(KG_{m,n})$ for arbitrary $m$ and $n$ seems hard. In the following theorem we establish $\mathcal{L}(KG_{2,n})$. The result makes use of the Catalan numbers, denoted by $C_n$ (e.g., \cite[Theorem 1.4.1]{St}) and provides the first combinatorial interpretation to \seqnum{A052712}.

\begin{theorem}\label{thm;1}
We have $\mathcal{L}(KG_{2,n})=2^{n-1}(n+1)!C_n$.
\end{theorem}

\begin{proof}
We begin by calculating the number of random walk labelings of $KG_{2,n}$ that start on the first column, denoted by $t_n$. We claim that $t_n=\frac{(2n)!}{n!}$ (\seqnum{A001813}). Indeed, it is clear that $t_1=2$ and, proceeding inductively, we consider a labeling of $KG_{2,n+1}$ that starts on the first column. There are two rows in which the label $1$ can be placed. Now, in any case, the two rows of the second column are reachable by the random walk and the subgraph of $KG_{2,n+1}$ consisting of the $n$ columns $2,3,\ldots,n+1$ is isomorphic to $KG_{2,n}$ and its labeling starts on its first column (the second column of $KG_{2,n+1}$). The unlabeled vertex in the first column of $KG_{2,n+1}$ can be labeled with any of the $2n+1$ labels $2,3,\ldots,2(n+1)$. Thus,
$$t_{n+1}=2(2n+1)t_n=2(2n+1)\frac{(2n)!}{n!}=\frac{(2n+2)!}{(n+1)!}.$$

Consider now an arbitrary random walk labeling of $KG_{2,n}$, where $n>1$. It may start at any of the $n$ columns and at either the first or the second row. As in the proof of the second result in Example \ref{ex;110}, if the labeling starts in the $k$th column, where $k\in[n]$, then the graph $KG_{2,n}$ is splitted into two graphs isomorphic to $KG_{2,k-1}$ and $KG_{2,n-k}$. The restriction of the labeling of $KG_{2,n}$ to these two subgraphs gives two labelings that start in the $k-1$th and $n-k$th column, respectively, which may be considered as the first columns of each of them. Finally, we may label the vertex vertically adjacent to the label $1$ with any of the $2n-1$ labels $2,3,\ldots,2n$. It follows that
\begin{align}
\mathcal{L}(KG_{2,n})&=2(2n-1)\sum_{k=1}^n\binom{2(n-1)}{2(k-1)}t_{k-1}t_{n-k}\nonumber\\&=2(2n-1)\sum_{k=1}^n\binom{2(n-1)}{2(k-1)}\frac{(2(k-1))!}{(k-1)!}\frac{(2(n-k))!}{(n-k)!}\nonumber\\ &=2\frac{(2(n-1))!}{(n-1)!}(2n-1)\sum_{k=1}^{n}\binom{n-1}{k-1}\nonumber\\
	&=2^{n-1}(n+1)!C_n\nonumber.\qedhere
\end{align}
\end{proof}

\subsection{The grid graph}
The \emph{grid graph} of size $m\times n$ is defined to be the graph whose vertex set $V$ is given by $$V=\left\{(i,j)\;:\;i\in[m],j\in[n]\right\}$$ and two vertices $(i_1,j_1),(i_2,j_2)\in V$ are adjacent if and only if $|i_1-i_2|+|j_1-j_2|=1$ (cf.~\cite[p.~194]{H}). We denote by $GG_{m,n}$ the grid graph of size $m\times n$. As in the case of
$\mathcal{L}(KG_{m,n})$, calculating $\mathcal{L}(GG_{m,n})$ for arbitrary $m$ and $n$ seems hard. In the following theorem we establish $\mathcal{L}(GG_{2,n})$.

\begin{theorem}\label{thm;521}
We have \begin{equation}\label{eq;612}\mathcal{L}(GG_{2,n})=2^{n-1}(n-1)!\sum_{k=0}^{n-1}\frac{n\binom{2(n-1)}{2k}+\binom{2n-1}{2k}}{\binom{n-1}{k}}.\end{equation}
Thus, asymptotically,  
$$\mathcal{L}(GG_{2,n})\sim \sqrt{\pi n}n!2^{2n-3}.$$
\end{theorem}

\begin{proof}
We begin by calculating the number of random walk labelings of $GG_{2,n}$, that start on the upper left vertex, denoted by $t_n$. We claim that $t_n=2^{n-1}n!$ (\seqnum{A002866}). Indeed, it is clear that $t_1=1$ and, proceeding inductively, we consider a labeling of $GG_{2,n+1}$ that starts on the upper left vertex. The remaining $n$ columns are isomorphic to $GG_{2,n}$. There are two vertices that can be labeled with $2$, namely, directly to the right or below label $1$. In the latter case, the walk can proceed in either the first or the second row of the second column, and the restriction of the labeling of $GG_{2,n+1}$ to $GG_{2,n}$ corresponds to a labeling of the latter that starts on its leftmost column. Thus, in this case, there are $2t_n$ labelings. In the former case, the labeling of $GG_{2,n}$ starts in the upper left vertex, and the vertex below the label $1$ can be labeled with any of the $2n$ remaining labels $3,4,\ldots,2(n+1)$. Thus, in this case, there are $2nt_n$ labelings. Combining the two cases, we have
$$t_{n+1}=2t_n+2nt_n=2(n+1)t_n=2(n+1)2^{n-1}n!=2^{n}(n+1)!.$$

Consider now an arbitrary random walk labeling of $GG_{2,n}$, where $n>1$. Suppose that it starts in the $k$th column, where $k\in[n]$.
We proceed now as in the proof of Theorem \ref{thm;1}, but we
distinguish between three cases, where we assume throughout that the label $1$ is in the upper row. Due to symmetry, the result will be multiplied in the end by $2$. It will be convenient to set $t_0:=1$.
\begin{enumerate}
\item [(a)] The label $2$ is also in the $k$th column (see Figure \ref{fig:10} below):
\begin{figure}[H]
\centering
\begin{tikzpicture}
\matrix[mystyle]{
$1$& $2$& $\cdots$&$k$&$\cdots$ & &&$\cdots$&$n$\\
& & &$1$& & &&&\\
& & &$2$& & &&&\\};
\end{tikzpicture}\caption{The label $2$ is also in the $k$th column.}\label{fig:10}
\end{figure}
In this case, the labeling of $GG_{2,k-1}$ begins in one of the two vertices of column $k-1$. The number of such labelings, for $k\geq 2$, is given by $2t_{k-1}$. Similarly, there are $2t_{n-k}$ labelings of $GG_{2,n-k}$, provided that $k\leq n-1$. There are $\binom{2(n-1)}{2(k-1)}$ different ways to distribute the remaining $2(n-1)$ labels $3,4,\ldots,2n$ between the two grid subgraphs. We conclude that, in this case, the number of labelings is given by
\begin{align}
&\overbrace{2t_{n-1}}^{k=1}+\sum_{k=2}^{n-1}\binom{2(n-1)}{2(k-1)}2t_{k-1}2t_{n-k}+\overbrace{2t_{n-1}}^{k=n}=\nonumber\\&\sum_{k=2}^{n-1}\binom{2(n-1)}{2(k-1)}2t_{k-1}2t_{n-k}+4t_{n-1}=\nonumber\\&2^{n-1}\sum_{k=2}^{n-1}\binom{2(n-1)}{2(k-1)}(k-1)!(n-k)!+2^n(n-1)!.\label{eq;a4}
\end{align}

\item [(b)]\label{bb} The labelings of the two grid subgraphs start in the upper row (see Figure \ref{fig:20} below for an example): \begin{figure}[H]
\centering
\begin{tikzpicture}
\matrix[mystyle]{
$1$& $2$& $\cdots$&$k$&$\cdots$ & &&$\cdots$&$n$\\
& &$5$ &$1$& $2$& &&&\\
& & && $3$&$4$ &&&\\};
\end{tikzpicture}\caption{The labelings of the two grid subgraphs start in the upper row.}\label{fig:20}
\end{figure}
Here, for every $1\leq k\leq n$, there are $t_{k-1}$ and $t_{n-k}$ labelings of the two grid subgraphs. Additionally, there are $2(n-1)$ possible labels for the vertex below the label $1$, namely $2,3,\ldots,2n$. Suppose it is the label $i$. Then there are $\binom{2(n-1)}{2(k-1)}$ different ways to distribute the remaining $2(n-1)$ labels $\{2,3,4,\ldots,2n\}\setminus\{i\}$ between the two grid subgraphs. We conclude that, in this case, the number of labelings is given by
\begin{align}
&2(n-1)\sum_{k=1}^n\binom{2(n-1)}{2(k-1)}t_{k-1}t_{n-k}=\nonumber\\&2^{n-2}(n-1)\sum_{k=2}^{n-1}\binom{2(n-1)}{2(k-1)}(k-1)!(n-k)!+2^n(n-1)(n-1)!\label{eq;a3}.\end{align}

\item [(c)] The label $2$ is in the upper row and the labeling of the other grid subgraph starts in the lower row (see Figure \ref{fig:30} below for an example):
\begin{figure}[H]
\centering
\begin{tikzpicture}
\matrix[mystyle]{
$1$& $2$& $\cdots$&$k$&$\cdots$ & &&$\cdots$&$n$\\
& &$2$ &$1$& & &&&\\
 &$4$ &$3$&$5$& $6$& &&&\\};
\end{tikzpicture}\caption{The label $2$ is in the upper row and the labeling of the other grid subgraph starts in the lower row.}\label{fig:30}
\end{figure}
We notice that the cases $k=1$ and $k=n$ are covered by part (\hyperref[bb]{b}). We may assume that the label $2$ is on the left of label $1$. Due to symmetry, the result will be multiplied by $2$. As before, there are $t_{k-1}$  and $t_{n-k}$ labelings of the two grid subgraphs. The labeling of the right grid subgraph cannot begin before the vertex beneath the label $1$ is labeled. Suppose that this vertex is labeled with $3\leq \ell\leq 2(k-1)+2$. Then there are $\binom{2n -\ell}{2(k-1)-(\ell-2)}$ different ways to distribute the remaining $2n-\ell$ labels $\{\ell+1,\ell+2,\ldots,2n\}$ between the two grid subgraphs. We conclude that, in this case, the number of labelings is given by
\begin{align}
&2\sum_{k=2}^{n-1}\sum_{\ell=3}^{2(k-1)+2}\binom{2n-\ell}{2(k-1)-(\ell-2)}t_{k-1}t_{n-k}=\nonumber\\&2^{n-2}\sum_{k=2}^{n-1}\sum_{\ell=3}^{2(k-1)+2}\binom{2n-\ell}{2(k-1)-(\ell-2)}(k-1)!(n-k)!.\label{eq;a1}
\end{align}
In order to simplify, let $2\leq k\leq n-1$. Then
\begin{align}
\sum_{\ell=3}^{2(k-1)+2}\binom{2n-\ell}{2(k-1)-(\ell-2)}&=\sum_{\ell=0}^{2(k-1)-1}\binom{2(n-1)-1-\ell}{2(k-1)-1-\ell}\nonumber\\&=\binom{2(n-1)}{2(k-1)-1}, \nonumber
\end{align} where the last equality is due to the Hockey-stick identity (e.g., \cite[(1) on p.~7]{R}). Thus, \begin{equation}\label{eq;a2}
(\ref{eq;a1}) = 2^{n-2}\sum_{k=2}^{n-1}\binom{2(n-1)}{2(k-1)-1}(k-1)!(n-k)!.
\end{equation}

\end{enumerate}
We may now add the three results and obtain
\begin{align}
&\mathcal{L}(GG_{2,n})\nonumber\\ &= 2((\ref{eq;a4})+(\ref{eq;a3})+(\ref{eq;a2}))\nonumber\\&=2^{n+1}n!+2^{n-1}\sum_{k=2}^{n-1}\left((n+1)\binom{2(n-1)}{2(k-1)}+\binom{2(n-1)}{2(k-1)-1}\right)(k-1)!(n-k)!\nonumber\\&=2^{n+1}n!+2^{n-1}\sum_{k=2}^{n-1}\left(n\binom{2(n-1)}{2(k-1)}+\binom{2n-1}{2(k-1)}\right)(k-1)!(n-k)!\nonumber\\&=2^{n-1}\sum_{k=1}^n\left(n\binom{2(n-1)}{2(k-1)}+\binom{2n-1}{2(k-1)}\right)(k-1)!(n-k)!\nonumber\\&=2^{n-1}(n-1)!\sum_{k=0}^{n-1}\frac{n\binom{2(n-1)}{2k}+\binom{2n-1}{2k}}{\binom{n-1}{k}},\nonumber
\end{align}
where in the third equality we used Pascal's rule.

We now wish to prove the claim regarding the asymptotic of $\mathcal{L}(GG_{2,n})$. In the proof of Theorem \ref{thm;880} we shall establish that the egf $A(x)$ of $\mathcal{L}(GG_{2,n})$ is given by
$$
A(x) = \frac{(1-2x)^{2}\arctan\left(\frac{2x}{\sqrt{1-4x}}\right)+2x\sqrt{1-4x}}{2\left(\sqrt{1-4x}\right)^{3}}.$$ By using the singularity analysis (see \cite[Chapter VI]{F}) and that the Laurent series of $A(x)$ at $x=1/4$ is given by
$$A(x)=\frac{\pi}{16(1-4x)^{3/2}}+O\left(\frac{1}{1-4x}\right),$$
we have 
\begin{align*}
\mathcal{L}(GG_{2,n})&\sim \sqrt{\pi n}n!2^{2n-3}.
\end{align*}
\end{proof}

\subsection{Some combinatorial identities}

Recall that, by Theorem
\ref{thm;521}, we have $$\mathcal{L}(GG_{2,n})=2^{n-1}(n-1)!\sum_{k=0}^{n-1}\frac{n\binom{2(n-1)}{2k}+\binom{2n-1}{2k}}{\binom{n-1}{k}}.$$
In this section, we analyze this expression, which involves the inverses of the binomial coefficients. A standard tool in this setting was introduced by Trif \cite{Tr}. It is based on Euler's Beta function $B(a,b)$ that, for $a,b>0$, is defined by $$B(a,b)=\int_0^1t^{a-1}(1-t)^{b-1}dt$$
and on the fact that, for $1\leq k\leq n$, we have $$\binom{n}{k}^{-1}=(n+1)B(k+1,n-k+1)=kB(k,n-k+1).$$
In our calculations, we were aided by Maple and whenever we introduce generating functions, we tacitly assume that $x$ is a positive small number, say, $0<x<0.01$.

Our first result is concerned with the expression for $\mathcal{L}(GG_{2,n})$ as a whole. It turns out that $\mathcal{L}(GG_{2,n})$ coincides with \seqnum{A087923}, which counts the number of ways of arranging the numbers $1,\ldots,2n$ into a $2\times n$ array, such that there is exactly one local maximum. G\"obel stated recently that $$\seqnum{A087923}(n)=2\sum_{k=1}^n \frac{(2n-2)!(2k(n-k+1)-1)}{(2k-1)!!(2n-2k-1)!!}.$$ Using the fact that, for odd $m=2k-1$ with $k\in\N$, we have $m!! = \frac{(2k)!}{2^kk!}$, we may write $$\seqnum{A087923}(n)=2^{n}(n-1)!\sum_{k=0}^{n-1}\frac{\binom{2(n-1)}{2k}(2(k+1)(n-k)-1)}{\binom{n-1}{k}(2k+1)}.$$

\begin{theorem}\label{thm;880}
We have \begin{equation}\label{eq;915}
2^{n-1}(n-1)!\sum_{k=0}^{n-1}\frac{n\binom{2(n-1)}{2k}+\binom{2n-1}{2k}}{\binom{n-1}{k}}=2^n(n-1)!\sum_{k=0}^{n-1}\frac{\binom{2(n-1)}{2k}(2(k+1)(n-k)-1)}{\binom{n-1}{k}(2k+1)}.
\end{equation}
Thus, $\mathcal{L}(GG_{2,n})=\seqnum{A087923}(n)$.
\end{theorem}
\begin{proof} Denote the left- (resp.\ right) hand side of (\ref{eq;915}) by  $a_n$ (resp.\ $b_n$) and let $A(x)$ and $B(x)$ be their corresponding egfs, respectively.
We have
\begin{align*}
&A(x)\nonumber\\&=\sum_{n\geq1}\frac{a_n}{n!}x^n=\sum_{n\geq1}\sum_{k=0}^{n-1}\int_0^12^{n-1}\left(n\binom{2(n-1)}{2k}+\binom{2n-1}{2k}\right)t^k(1-t)^{n-1-k}x^ndt\\ &=\int_0^1\sum_{n\geq1}\sum_{k\geq0}2^{n+k-1}\left((n+k)\binom{2(n+k-1)}{2k}+\binom{2n+2k-1}{2k}\right)t^k(1-t)^{n-1}x^{n+k}dt\\
&=\int_0^1\frac{2x(4\left(8t^{3}-12t^{2}+6t-1\right)x^{3}-8\left(t-1\right)x^{2}+2tx-5x+1)}{(16t\left(t-1\right)x^{2}+(1-2x)^{2})^{2}}dt.
\end{align*}
Similarly,
\begin{align*}
&B(x)\nonumber\\&=\sum_{n\geq1}\frac{b_n}{n!}x^n=\sum_{n\geq1}\sum_{k=0}^{n-1}\int_0^1\frac{2^n\binom{2(n-1)}{2k}(2(k+1)(n-k)-1)}{2k+1}t^k(1-t)^{n-1-k}x^ndt\\
&=\int_0^1\sum_{n\geq1}\sum_{k\geq0}\frac{2^{n+k}\binom{2(n+k-1)}{2k}(2n(k+1)-1)}{2k+1}t^k(1-t)^{n-1}x^{n+k}dt\\
&=-\int_0^1\frac{2x(1-2x)^{2}(4tx-2x-1)}{(16t\left(t-1\right)x^{2}+(1-2x)^{2})^{2}}dt.
\end{align*}
Thus,
$$A(x)=\frac{(1-2x)^{2}\arctan\left(\frac{2x}{\sqrt{1-4x}}\right)+2x\sqrt{1-4x}}{2\left(\sqrt{1-4x}\right)^{3}}=B(x)$$ and therefore
$a_n=b_n$.
\end{proof}

The following two results are concerned with the second summand in the expression for $\mathcal{L}(GG_{2,n})$. In Theorem \ref{thm;434} we prove that $(n-1)!\sum_{k=0}^{n-1}\binom{2n-1}{2k}\binom{n-1}{k}^{-1}$ coincides with \seqnum{A087547}, that is defined by
\begin{equation}\label{eq;881}
\seqnum{A087547}(n)=n!2^{n+1}\int_{0}^{1}\frac{1}{\left(1+x^{2}\right)^{n+1}}dx-\frac{\pi(2n)!}{2^{n+1}n!}.
\end{equation}
Bala stated that \begin{equation}\label{eq;900}
\seqnum{A087547}(n)= \frac{(2n)!}{n!2^n}\sum_{k = 0}^{n-1} \frac{2^k(k!)^2}{(2k+1)!}
\end{equation} and conjectured that \begin{equation}\label{eq;901}
\seqnum{A087547}(n)=\sum_{k=1}^{n}\frac{2^{k-1}(k-1)!(n+k-1)!}{(2k-1)!}.
\end{equation}
In Theorem \ref{thm;001} we prove that Bala's conjecture is true.

\begin{theorem}\label{thm;434}
We have \begin{equation}\label{eq;771}
\seqnum{A087547}(n)=(n-1)!\sum_{k=0}^{n-1}\frac{\binom{2n-1}{2k}}{\binom{n-1}{k}}.
\end{equation}
\end{theorem}
\begin{proof}
Denote the left- (resp.\ right) hand side of (\ref{eq;771}) by  $a_n$ (resp.\ $b_n$). We have
\begin{align*}
\sum_{n\geq1}\frac{b_n}{(n-1)!}x^n&=\sum_{n\geq1}\sum_{k=0}^{n-1}\int_0^1n\binom{2n-1}{2k}t^k(1-t)^{n-1-k}x^ndt\\
&=\int_0^1\sum_{n\geq1}\sum_{k\geq0}(n+k)\binom{2n+2k-1}{2k}t^k(1-t)^{n-1}x^{n+k}dt\\
&=-\int_0^1\frac{x(4t^2x^2-4tx-(1-x)^2)}{(4t^2x^2-4tx^2+(1-x)^2)^2}dt\\
&=\frac{x\left((1-x)\arctan\left(\frac{x}{\sqrt{1-2x}}\right)+\sqrt{1-2x}\right)}{(1-x)\left(\sqrt{1-2x}\right)^3}.
\end{align*}
On the other hand, it is known that $(a_n)_{n\in\N}$ satisfies the recursion formula $a_n=(2n-1)a_{n-1}+(n-1)!$ with $a_1=1$. Thus,
$$\sum_{n\geq1}\frac{a_n}{(n-1)!}x^n=\frac{x\left((1-x)\arctan\left(\frac{x}{\sqrt{1-2x}}\right)+\sqrt{1-2x}\right)}{(1-x)\left(\sqrt{1-2x}\right)^3}.$$
Hence, $a_n=b_n$.
\end{proof}

\begin{theorem}\label{thm;001}
We have
\begin{equation}\label{eq;003}
\sum_{k=0}^{n-1}\frac{2^{k}}{(2k+1)\binom{2k}{k}}\binom{n+k}{k}=\frac{\binom{2n}{n}}{2^{n}}\sum_{k=0}^{n-1}\frac{2^{k}}{(2k+1)\binom{2k}{k}}.
\end{equation}
In particular, Bala's conjecture is true.
\end{theorem}

\begin{proof}
Writing
$$(\ref{eq;900}) = \binom{2n}{n}\frac{n!}{2^n}\sum_{k=0}^{n-1}\frac{2^{k}}{(2k+1)\binom{2k}{k}}$$
and
$$(\ref{eq;901})= \sum_{k=0}^{n-1}\frac{2^{k}k!(n+k)!}{(2k+1)!},$$
we see that proving the identity in (\ref{eq;003}) settles Bala's conjecture.

Let us denote the left- (resp.\ right) hand side of (\ref{eq;003}) by  $a_n$ (resp.\ $b_n$) and let $A(x)$ and $B(x)$ be their corresponding generating functions, respectively. We have
\begin{align*}
A(x)&=\sum_{n\geq1}a_nx^n=\sum_{n\geq1}\sum_{k=0}^{n-1}\int_0^12^k\binom{n+k}{k}t^k(1-t)^kx^ndt\\
&=\int_0^1\sum_{n\geq1}\sum_{k\geq0}2^k\binom{n+2k}{k}t^k(1-t)^kx^{n+k}dt\\    &=\int_0^1\frac{2x}{\sqrt{1-8xt(1-t)}\left(1-2x+\sqrt{1-8xt(1-t)}\right)}dt.
\end{align*}
Similarly,
\begin{align*}
&B(x)\nonumber\\&=\sum_{n\geq1}b_nx^n=\sum_{n\geq1}\sum_{k=0}^{n-1}\int_0^1\frac{\binom{2n}{n}}{2^{n-k}}t^k(1-t)^kx^ndt\\&=\int_0^1\frac{2x\left((2t^2-2t)\sqrt{1-2x}+2t^2-2t+1+\sqrt{4t^2x-4tx+1}\right)}
{(2t^2-2t+1)\sqrt{4t^2x-4tx+1}\left(1+\sqrt{4t^2x-4tx+1}\right)\sqrt{1-2x}\left(1+\sqrt{1-2x}\right)}dt,
\end{align*}
Thus,
\begin{align*}
A(x)=\frac{\pi-4\arctan\left(\sqrt{1-2x}\right)}{2\sqrt{1-2x}}=B(x)
\end{align*}
and therefore $a(n)=b(n)$.
\end{proof}

The first summand in the expression for $\mathcal{L}(GG_{2,n})$ defines \seqnum{A182525}, i.e., \begin{equation}\label{eq;700}
\seqnum{A182525}(n)=n!\sum_{k=0}^{n}\frac{\binom{2n}{2k}}{\binom{n}{k}}.
\end{equation}

In the following theorem, we establish the egf of the sequence.

\begin{theorem}
The egf $A(x)$ of \seqnum{A182525}  is given by $$A(x)=\frac{x\arctan\left(\frac{x}{\sqrt{1-2x}}\right)+\sqrt{1-2x}}{\left(\sqrt{1-2x}\right)^{3}}.$$
\end{theorem}

\begin{proof}
Denote by $a_n$ the right-hand side of (\ref{eq;700}). We have
\begin{align*}
\pushQED{\qed}
A(x)&=\sum_{n\geq0}\frac{a_n}{n!}x^n=\sum_{n\geq0}\sum_{k=0}^{n}\int_0^1(n+1)\binom{2n}{2k}t^k(1-t)^{n-k}x^ndt\\
&=\int_0^1\sum_{n\geq0}\sum_{k\geq0}(n+1+k)\binom{2(n+k)}{k}t^k(1-t)^{n}x^{n+k}dt\\    &=\int_0^1\frac{(x-1)^{2}-4t(t-1)x^{2}}{((x-1)^{2}+4t(t-1)x^{2})^{2}}dt\\&=\frac{x\arctan\left(\frac{x}{\sqrt{1-2x}}\right)+\sqrt{1-2x}}{\left(\sqrt{1-2x}\right)^{3}}.\qedhere
\end{align*}
\end{proof}

We conclude our work by establishing two identities in the spirit of (\ref{eq;881}).
\begin{lemma}\label{lem;52}
We have \begin{align}
\sum_{k=0}^{n}\frac{\binom{2n+1}{2k}}{\binom{n}{k}}&=1+\frac{2n+1}{2}\int_{0}^{\pi/2}\left(1+\sin(2t)\right)^{n}-\left(1-\sin(2t)\right)^{n}dt,\nonumber\\\sum_{k=0}^{n}\frac{\binom{2n}{2k}}{\binom{n}{k}}&=1+n\int_{0}^{\pi/2}\cos t\left(\left(\cos t+\sin t\right)^{2n-1}-\left(\cos t-\sin t\right)^{2n-1}\right)dt.\nonumber\end{align}
\end{lemma}

\begin{proof}
We prove only the first identity, the second being similar. We have
\begin{align}
\sum_{k=0}^{n}\frac{\binom{2n+1}{2k}}{\binom{n}{k}}&=1+\sum_{k=1}^{n}\binom{2n+1}{2k}kB(k,n+1-k)\nonumber\\&=1+\sum_{k=1}^{n}\binom{2n+1}{2k}k\int_{0}^{1}x^{k-1}(1-x)^{n-k}dx\nonumber\\&=1+\frac{2n+1}{2}\int_{0}^{1}\sum_{k=1}^{n}\binom{2n}{2k-1}x^{k-1}(1-x)^{n-k}dx\nonumber\\&=1+\frac{2n+1}{2}\int_{0}^{1}(1-x)^{n-1}\sum_{k=0}^{n-1}\binom{2n}{2k+1}\left(\frac{x}{1-x}\right)^{k}dx\nonumber\\&=1+\frac{2n+1}{2}\int_{0}^{1}(1-x)^{n-1}\frac{\left(1+\sqrt{\frac{x}{1-x}}\right)^{2n}-\left(1-\sqrt{\frac{x}{1-x}}\right)^{2n}}{2\sqrt{\frac{x}{1-x}}}dx\nonumber\\&=1+\frac{2n+1}{2}\int_{0}^{1}\frac{\left(\sqrt{1-x}+\sqrt{x}\right)^{2n}-\left(\sqrt{1-x}-\sqrt{x}\right)^{2n}}{2\sqrt{x(1-x)}}dx\nonumber\\&=1+\frac{2n+1}{2}\int_{0}^{\pi/2}\left(\cos t+\sin t\right)^{2n}-\left(\cos t-\sin t\right)^{2n}dt\nonumber\\&=1+\frac{2n+1}{2}\int_{0}^{\pi/2}\left(1+\sin(2t)\right)^{n}-\left(1-\sin(2t)\right)^{n}dt\nonumber
\end{align} where in the fifth equality we used that, for $n\geq 1$, $$\sum_{k=0}^{\left\lfloor \frac{n-1}{2}\right\rfloor }\binom{n}{2k+1}x^{k}=\frac{(1+\sqrt{x})^{n}-(1-\sqrt{x})^{n}}{2\sqrt{x}} $$ (e.g., \cite[(1.95)]{Gou}).
\end{proof}

\end{document}